\documentclass[11pt,a4paper]{article}

\usepackage[T1]{fontenc}
\usepackage[utf8]{inputenc} 
\usepackage{lmodern}        
\usepackage{textcomp}       
\usepackage{microtype}
\usepackage[unicode]{hyperref}

\title{Space of Timelike Directions and Curvature Bounds}
\author{%
  Joe Barton\thanks{\texttt{joe.barton@univie.ac.at}} \and
  Jona R\"{o}hrig\thanks{%
    \texttt{jona.roehrig@univie.ac.at}
  }%
}

\usepackage{amsmath,bm,amsfonts,amstext,amsxtra,amssymb,amsthm,mathtools} 
\usepackage{tensor}

\usepackage{hyperref}
\usepackage{wrapfig}
\usepackage{graphics}
\usepackage{graphicx}
\usepackage{subfiles}
\usepackage{MnSymbol}
\usepackage{graphbox}
\usepackage{witharrows}
\usepackage{xcolor}
\usepackage{tikz}
\usetikzlibrary{tikzmark}
\usetikzlibrary{angles,quotes}
\usepackage{pgfplots} 
\usetikzlibrary{intersections, calc}
\usepgfplotslibrary{fillbetween}
\usetikzlibrary{patterns}
\pgfkeys{/pgf/declare function={arcsinh(\x) = ln(\x + sqrt(\x^2+1));}}
\allowdisplaybreaks

\usepackage[utf8]{inputenc} 
\usepackage[T1]{fontenc}

\usepackage[english]{babel} 

\usepackage[margin=1.5in]{geometry}

\usepackage{tikz}
\usetikzlibrary{matrix,quotes,calc,patterns,math,shapes}

\usepackage{float} 
\usepackage{verbatim}
\usepackage{thmtools}
\usepackage{caption}
\usepackage{subcaption}
\usepackage{wrapfig}
\usepackage{graphicx} 
\usepackage{enumitem}
\usepackage{mathrsfs}
\usepackage{stmaryrd}
\usepackage{lipsum}
\usepackage{bbm}

\usepackage{media9}
\definecolor{green}{HTML}{007F00}
\definecolor{red}{HTML}{7f0f0f}
\usepackage[makeroom]{cancel}
 \usepackage[T1]{fontenc}
  \usepackage{tikz-cd}

\usepackage{ifthen}
\usepackage{witharrows}
\usepackage{biblatex}
\usepackage{makeidx}
\makeindex
\usepackage[export]{adjustbox}
\hypersetup{
    colorlinks=true, 
    breaklinks=true, 
    citecolor=red, 
    linkcolor=black, 
    menucolor=red, 
    urlcolor=blue,
    bookmarksopen=false, 
    bookmarksopenlevel=0,
    pdftitle={Complete Minimal Surfaces and How to Print Them},
    pdfsubject={Complete Minimal Surfaces and How to Print Them},
    plainpages=false,
    hypertexnames=false,
}

\newcommand{\R}{\mathbb{R}} 

\newcommand{\dd}{\text{d}} 

\newtheorem{definition}{Definition}[section]

\newtheorem{example}[definition]{Example}

\newtheorem{lemma}[definition]{Lemma}

\newtheorem{theorem}[definition]{Theorem}

\theoremstyle{plain} 
\newtheorem{remark}{Remark}

\newcommand\defindex[2][]{\textbf{#2}\ifthenelse{\equal{#1}{}}{\index{#2}}{\index{#1}}} 

\newcommand{\eqval}{\; \Leftrightarrow \;}
\newcommand{\imply}{\; \Rightarrow \;}

\newcommand{\mDef}{\coloneqq }

\newcommand\identity{1\kern-0.25em\text{l}}
\DeclareMathOperator{\arcosh}{arcosh}

\newcommand{\LLS}{Lorentzian length space}

\newcommand{\LpLS}{Lorentzian pre-length space}

\newcommand{\Xll}{$(X,d,\ll,\leq,\tau)$ }

\usepackage{amsthm}
\usepackage{thmtools}
\usepackage{cleveref}

\pgfplotsset{compat=1.18} 
\begin{document}

\maketitle

\begin{abstract}
We investigate the consequences of timelike sectional curvature bounds in Lorentzian length spaces for the existence and structure of the space of directions at a point. It is established that, under upper timelike sectional curvature bounds, the space of directions exists and is itself a metric space with curvature bounded above by $-1$. Furthermore, the metric cone over the space of directions, which canonically models the tangent space at a given point, is shown to constitute a Lorentzian length space with timelike sectional curvature bounded above by $0$. To do this, we introduce the notion of $\epsilon$-$\mu$ timelike sectional curvature bounds, which are compatible with pre-existing synthetic curvature conditions \cite{secCurvBouns}. These results extend the comparison-geometric framework to the Lorentzian setting, providing a synthetic characterization of geodesics, tangent cones, and curvature under causal constraints.
\end{abstract}

\tableofcontents

\section{Introduction}

In Riemannian geometry, the tangent plane at a point plays an important role in defining curvature via second-order approximations. In metric geometry, this notion of a tangent plane does not exist; however, with an adapted notion of a tangent plane, one should obtain similar structures. The synthetic \emph{analogue} in metric spaces is the space of directions, formed from equivalence classes of geodesics starting at a point; and the tangent cone, the cone over the space of directions, which models the local geometry as a blow-up around a point.

In 2018, Kunzinger and S\"amann \cite{KS:18} introduced \emph{\LpLS s} (LpLS), a Lorentzian version of a metric space. This synthetic Lorentzian geometry provides tools to study non-smooth Lorentzian spaces without relying on a differentiable structure, enabling comparison-based characterisations of geodesics, tangent cones and curvature in settings where smooth techniques may fail \cite{KS:18}. Recent synthetic approaches adapt comparison geometry from the metric to the Lorentzian context, aiming to obtain curvature and singularity results in a purely Lorentzian metric framework. \cite{barton2026splittingtheoremnonpositivelycurved}
\cite{alexander2021generalizedconeslorentzianlength}

The aim of this paper is to investigate the structure of tangent cones and their curvature in Lorentzian pre-length spaces and to understand how curvature bounds behave under this blow-up process. In the metric setting, there are results relating metric spaces with curvature bounded above to their spaces of directions and tangent cones, for example: if $X$ is a complete locally compact length space with curvature bounded above, then at each point $p \in X$, the tangent cone $T_p$ is a length space of non-positive curvature, and the space of directions at $p$ is a space of curvature $\leq 1$ \cite[Theorem 9.1.44]{burago2001course}. We present results that relate the behaviour of upper timelike curvature bounds in Lorentzian pre-length spaces to their tangent cones and spaces of directions.

In Section 2, we introduce $\epsilon$–$\mu$ midpoints, a notion of curvature bounds for non-geodesic Lorentzian length spaces that, under mild assumptions, are compatible with the four-point condition and triangle based bounds when geodesics exist.

In Section 3, we introduce the space of timelike directions and show in \Cref{them:SoD_lengthspace} that if $X$ is a Lorentzian length space with timelike sectional curvature bounded above, then the space of directions is a length space and the tangent cone is a Lorentzian pre-length space.

In Section 4, we prove that under upper timelike sectional curvature bounds, the future tangent cone $T_p^+$ satisfies a four-point curvature bound $\leq 0$ (hence it is a Lorentzian length space with non-positive timelike curvature in the comparison sense). Via a cone–base correspondence, this implies, in our main \Cref{thm:SoDcba-1} that the space of directions $\Sigma_p^+$ has curvature bounded above by $-1$ and it is in fact geodesic if the Lorentzian pre-length space has any upper timelike sectional curvature bound. 

These results extend classical comparison statements about spaces of directions and tangent cones to the Lorentzian setting \cite{burago2001course}.

\textbf{Acknowledgements:}
We would like to thank Clemens Sämann and Tobias Beran for their helpful discussions in the initial stages of this project. 

The authors were funded by the Austrian Science Fund (FWF) [Grant DOI's: 10.55776/EFP6]. 

For open access purposes, the authors have applied a CC BY public copyright license to any author-accepted manuscript version arising from this submission.

\section{Sectional Curvature Bounds}
In this section, we review several notions of timelike sectional curvature bounds for Lorentzian pre-length spaces, which will be needed in the later discussion. In addition, we introduce a new notion of curvature bounds, the \(\epsilon\)-\(\mu\) curvature bound, which does not require the local existence of geodesics. A good overview of various notions of timelike sectional curvature bounds can be found in \cite{secCurvBouns}.\\
The theory of synthetic timelike sectional curvature bounds is built upon the notion of Lorentzian metric spaces, such as \LLS s and \LpLS s. We will not repeat all necessary definitions here; instead, we refer the reader to \cite{KS:18} and only recall the most fundamental object, namely the \LpLS:

\begin{definition}
A \textbf{Lorentzian pre-length space} $(X,\dd,\ll, \leq, \tau)$ consists of a set $X$, a metric $\dd$ (positive definite, symmetric, and satisfying the triangle inequality), two relations $\ll, \leq\ \subset X\times X$ and a function $\tau: X\times X \to [ 0,\infty]$, such that: \\
$ 
\begin{array}{ll}
\begin{array}{l}
\bullet\ \leq \text{ is reflexive and transitive (a pre-order)} \\
\bullet\ \ll \text{ is transitive and contained in }\leq
\end{array}
& \!\!\!\!\bigg\}\ (X,\ll,\leq) \!\text{ is causal space}\\
\begin{array}{l}
\bullet\ \ \tau \text{ is lower semi-continuous w.r.t d}\\
\bullet\ \ \tau(x,z) \geq \tau(x,y) + \tau(y,z) \ \ \ \forall\ x\leq y \leq z \in X\\
\bullet\ \ x\nleq y \imply \tau(x,y)=0 \text{ and } \tau(x,y)>0 \eqval x\ll y
\end{array}
&\\
\end{array}
$
\end{definition}
\subsection{Previous notions of curvature bounds}
All definitions of sectional curvature are based on the fundamental idea of comparing time separations in a Lorentzian length space to time separations in 2-dimensional spaces of constant sectional curvature $K$, $\mathbb{L}^2_K$, called the $K$ - comparison spaces. For $K<0$, the comparison space is a scaled anti de-Sitter space; for $K=0$, Minkowski space; and for $K>0$, a scaled de-Sitter space. These spaces have diameter bounds $D_K = \frac{\pi}{\sqrt{-K}}$ for $K<0$ and $D_K = \infty$ for $K \geq 0$.
\begin{definition}
    Let \Xll be a \LpLS \, and $K\in \R$. For any three causally related points $x\leq y\leq z\in X$ such that $\tau(x,z)< D_K$, we can construct its \textbf{comparison triangle} $\triangle\bar x\bar y\bar z\subset \mathbb{L}^2_K$ in the 2-dimensional Lorentzian space of constant curvature $K$ such that all the side lengths agree:
    $$\tau(x,y) = \tau^{\mathbb{L}^2_K}(\bar x, \bar y), \quad \tau(x,z) = \tau^{\mathbb{L}^2_K}(\bar x, \bar z), \quad \tau(y,z) = \tau^{\mathbb{L}^2_K}(\bar y, \bar z).$$
    If the geodesic $\gamma_{xy}$ between $x$ and $y$ exists in $X$ (or any other combination of points), the comparison point of a point $m\in \gamma_{xy}$ is a point $\bar m\in \mathbb{L}^2_K$ on $\gamma_{\bar x \bar y}$ such that
    $$\tau(x,m) = \tau^{\mathbb{L}^2_K}(\bar x, \bar m) \quad \text{and} \quad \tau(m,y) = \tau^{\mathbb{L}^2_K}(\bar m, \bar y).$$
\end{definition} 
The one-sided triangle comparison defines curvature bounds by comparing distances in timelike triangles. 

\begin{definition}\label{def:oneSidedLorTriangle}
\textbf{One-sided triangle comparison} \cite[Def. 3.2]{secCurvBouns} \\ 
A \LpLS \, \Xll has an upper (or lower) curvature bound $K\in \R$ if it can be covered by comparison neighbourhoods $U\subset X$ which satisfy:
\begin{enumerate}
    \item $\tau$ is continuous on $(U\times U)\cap \tau^{-1}\bigl([0,D_K)\bigr)$,
    \item All points $x\leq y\in U$ can be connected by a distance realizer contained in $U$.
    \item For any causal triangle $\triangle xyz \subset U$, any point on one of the distance realizers of the triangle (for example, $m\in \gamma_{yz}$), and any comparison triangle $\triangle \bar x\bar y\bar z\subset \mathbb{L}^2_K$ with $\bar m\in \gamma_{\bar y \bar z}$, we have:
\end{enumerate}
\begin{equation*}
    \tau(x,m) \geq \tau_{\mathbb{L}^2_K}(\bar x, \bar m) 
    \quad 
    \bigl(\text{respectively, } \tau(x,m) \leq \tau_{\mathbb{L}^2_K}(\bar x, \bar m)\bigr).
\end{equation*}
\end{definition}

A different approach to defining curvature bounds uses angles and comparison angles in Lorentzian length spaces. The idea of a comparison angle is to transfer the side lengths of a triangle to a comparison space and calculate the hyperbolic angle between the geodesics connecting the vertices in that space:

\begin{definition} \textbf{\boldmath{$K$}-comparison angles} 
\cite[Def. 2.1]{HyperbolicAnglesBeran_2023} \\
Let \Xll be a \LpLS \, and $K\in \R$. For any three causally related points $x\leq y\leq z\in X$ such that $\tau(x,z)<D_K$ for $K>0$, the \textbf{\boldmath{$K$}-comparison angle} is defined as
\begin{equation*}
    \tilde{\measuredangle}_{x}^Kyz \mDef
     \measuredangle^{\mathbb{L}^2_K}_{\bar x }\bar y \bar z.
\end{equation*}
Here, $\measuredangle^{\mathbb{L}^2_K}_{\bar x}\bar y \bar z$ denotes the angle in the comparison triangle in $\mathbb{L}^2_K$ at $\bar x$ between the geodesics connecting $\bar x\bar y$ and $\bar x\bar z$. The angles at $\bar y$ and $\bar z$ are defined analogously. The quantity $\tilde{\measuredangle}_{x }^K yz$ is understood to be well-defined for any causal ordering of $x,y,z$. 

The \textbf{signed $\bm{K}$-comparison angle} is defined as
\[
\tilde{\measuredangle}_{x}^{K,S}yz \mDef \sigma\, \tilde{\measuredangle}_{x}^K yz,
\]
where the sign $\sigma=-1$ if $x$ is the most future or most past point, and $\sigma=1$ otherwise.
\end{definition}

\begin{remark}\label{rem:comparisonAngle}
The comparison angle is a function that depends only on $K$ and the three side lengths 
$a \mDef \tau(x,y)$, $b \mDef \tau(y,z)$, and $c \mDef \tau(x,z)$. Let $s \mDef \frac{1}{\sqrt{|K|}}$, then $K$-comparison angle is:
\begin{align*}
    \tilde{\measuredangle}_{x}^K yz &= \text{arccosh}\left(\frac{c^2 - a^2 + b^2}{2ab}\right) && \text{for } K = 0, \\
    \tilde{\measuredangle}_{x}^Kyz &= \text{arccosh}\left(\frac{-\cos(sa)\cos(sb) + \cos(sc)}{\sin(sa)\sin(sb)}\right) && \text{for } K < 0, \\
    \tilde{\measuredangle}_{x}^Kyz &= \text{arccosh}\left(\frac{\cosh(sa)\cosh(sb) - \cosh(sc)}{\sinh(sa)\sinh(sb)}\right) && \text{for } K > 0.
\end{align*}
\end{remark}

We can use the comparison angle for three points to define the angle between two causal curves (especially geodesics) starting from the same point:

\begin{definition} \textbf{Angles between curves} \cite[Def. 2.9]{HyperbolicAnglesBeran_2023}

Let \Xll be a \LpLS and let $\gamma_1,\gamma_2:[0,\epsilon)\to X$ be two causal curves starting at $p=\gamma_1(0)=\gamma_2(0)$ (each curve may point in either the future or past direction). 
The angle, or signed angle (with a negative sign if the time directions of the two curves agree), between $\gamma_1$ and $\gamma_2$ is defined as
\begin{equation*}
    \measuredangle_p\bigl(\gamma_1, \gamma_2\bigr) \mDef \limsup_{(t,s)\in A_K,t,s\to 0}\,\tilde{\measuredangle}^K_{p}\gamma_1(t)\gamma_2(s)
    \qquad \text{for } K\in \R,
\end{equation*}
 where $A_K$ is the set of all $t$ and $s$ such that $\triangle p \gamma_1(t)\gamma_2(s)$ is a timelike triangle satisfying the size bound for $\mathbb{L}^2_K$.
\end{definition}

\begin{remark}
The $\limsup$ in the angle definition is not necessarily a $\lim$, not even between a curve and itself. If the limit in the definition above exists, it is independent of the choice of $K\in \R$. This can be verified using the explicit formula for the comparison angles in \Cref{rem:comparisonAngle}.
\end{remark}
We now present two similar definitions for sectional curvature bounds, the \textbf{monotonicity condition} and the \textbf{angle comparison condition}. The definitions can be found in 3.8, 3.10 and 3.11 of \cite{secCurvBouns}. 

\begin{definition}\label{def:MonotonLorCurvature}
\textbf{Monotonicity comparison} \\
A \LpLS \, \Xll has an upper (or lower) curvature bound $K\in \R$ if it is covered by comparison neighbourhoods $U\subset X$ which satisfy:
\begin{enumerate}
    \item $\tau$ is continuous on $(U\times U)\cap \tau^{-1}\bigl([0,D_K)\bigr)$.
    \item All points $x\leq y\in U$ can be connected by a distance realizer contained in $U$.
    \item For any timelike distance realizers $\gamma_1,\gamma_2\in U$ starting from the same point $p$, the function $\theta(s,t)\mDef \tilde{\measuredangle}^{K,S}_{p}\gamma_1(t)\gamma_2(s)$ is monotonically decreasing (resp. increasing) for all $(t,s)\in A_K$.
\end{enumerate}
\end{definition}
\begin{remark}
    We note that this definition is for the signed angles. In the case we will consider, both curves are future pointing, and the signed angle is therefore negative. This means that the unsigned angles behave exactly the opposite way. 
\end{remark}
\begin{remark} \label{rem:anglecomp}
The definitions show that the angles between geodesics always exist if the space has either an upper or a lower curvature bound. Since the angle between two curves is defined via the limit supremum of $\theta$, upper curvature bounds imply that
\[
\measuredangle_p^S\bigl(\gamma_1,\gamma_2\bigr)\,\geq\, \tilde{\measuredangle}^{K,S}_{p}\gamma_1(t)\gamma_2(s),
\]
while lower curvature bounds imply
\[
\measuredangle_p^S\bigl(\gamma_1,\gamma_2\bigr)\,\leq\, \tilde{\measuredangle}^{K,S}_{p}\gamma_1(t)\gamma_2(s).
\]
\end{remark}

We now introduce the four point condition, a notion of sectional curvature bounds that requires the least amount of regularity on our spaces, since it does not rely on the local existence of geodesics. Instead, it relies solely on comparing the distances between four points with the corresponding distances in the comparison space.

\begin{definition}\textbf{Four point condition} \cite[Def. 4.6]{secCurvBouns}, \cite[Def. 4.1]{beran2025reshetnyakmajorisationdiscreteupper}
\label{def:LorfourPts} \\
A \LpLS \, \Xll has \textbf{curvature bounded below} by $K\in \R$, if it can be covered by neighbourhoods $U\subset X$ such that:
\begin{enumerate}
\item $\tau$ is continuous on $(U\times U)\cap\tau^{-1}([0,D_K)])$ and this set is open.
\item For any four points $x_1\ll x_2\ll x_3\leq x_4\subset U$, after constructing the comparison four point picture below in $\mathbb{L}^2_K$, we find that $\tau(x_2,x_3)\geq \tau(\tilde{x}_2,\tilde{x}_3)$.
\end{enumerate}

\begin{center}
\begin{tikzpicture}[scale=3]
\coordinate (A1) at (0,0);
\coordinate (B1) at (0,0.5);
\coordinate (C1) at (-0.2,1);
\coordinate (D1) at (0.3,1.4);
\draw (A1) -- (B1) -- (D1) -- cycle;
\draw (A1) -- (B1) -- (C1) -- cycle;
\draw[dashed, red] (C1) -- (D1);
\node[fill=black, inner sep=1pt, circle, label=left:$\tilde{x}_1$] at (A1) {};
\node[fill=black, inner sep=1pt, circle, label=above:$\tilde{x}_2$] at (B1) {};
\node[fill=black, inner sep=1pt, circle, label=left:$\tilde{x}_3$] at (C1) {};
\node[fill=black, inner sep=1pt, circle, label=right:$\tilde{x}_4$] at (D1) {};

\end{tikzpicture}
\end{center}

A \LpLS \Xll has \textbf{curvature bounded above} by $K\in \R$, if it can be covered by neighbourhoods $U\subset X$ such that:
\begin{enumerate}
\item $\tau$ is continuous on $(U\times U)\cap\tau^{-1}([0,D_K)])$ and this set is open.
\item For any four points $x_1\ll x_2\ll x_3\ll x_4\subset U$, after constructing the two comparison four point pictures below in $\mathbb{L}^2_K$, we find that $\tau(x_2,x_3)\geq \tau(\tilde{x}_2^1,\tilde{x}_3^1)$ and $\tau(x_1,x_4)\leq \tau(\tilde{x}_1^2,\tilde{x}_4^2)$.
\end{enumerate}
\begin{center}
\begin{tikzpicture}[scale=3]
\coordinate (A1) at (0,0);
\coordinate (B1) at (-0.1,0.5);
\coordinate (C1) at (0.2,1);
\coordinate (D1) at (0.1,1.4);
\draw (A1) -- (B1) -- (D1) -- (C1) -- cycle;
\draw[dashed, red] (B1) -- (C1);
\draw (A1) -- (D1);
\node[fill=black, inner sep=1pt, circle, label=left:$\tilde{x}_1^1$] at (A1) {};
\node[fill=black, inner sep=1pt, circle, label=left:$\tilde{x}_2^1$] at (B1) {};
\node[fill=black, inner sep=1pt, circle, label=right:$\tilde{x}_3^1$] at (C1) {};
\node[fill=black, inner sep=1pt, circle, label=left:$\tilde{x}_4^1$] at (D1) {};

\coordinate (A2) at (1,0);
\coordinate (B2) at (0.9,0.5);
\coordinate (C2) at (0.9,1);
\coordinate (D2) at (1.1,1.4);
\draw (A2) -- (B2) -- (C2) -- (D2);
\draw (B2) -- (D2);
\draw (A2) -- (C2);
\draw[dashed, red] (A2) -- (D2);
\node[fill=black, inner sep=1pt, circle, label=right:$\tilde{x}_1^2$] at (A2) {};
\node[fill=black, inner sep=1pt, circle, label=left:$\tilde{x}_2^2$] at (B2) {};
\node[fill=black, inner sep=1pt, circle, label=left:$\tilde{x}_3^2$] at (C2) {};
\node[fill=black, inner sep=1pt, circle, label=right:$\tilde{x}_4^2$] at (D2) {};
\end{tikzpicture}
\end{center}
\end{definition}

\subsection{\texorpdfstring{$\epsilon$–$\mu$}{epsilon-mu} sectional curvature bounds}
Next, we introduce a new notion of curvature that lies between the triangle comparison and the four point condition. The space must be a length space but does not need to be locally geodesic. Instead of comparing distances to the midpoints of geodesics (which do not necessarily exist), we compare distances to points that are almost in between two other points. This notion of curvature (for non-geodesic spaces) integrates better with the (Riemannian) length space setting and is inspired by \cite[Def. 9.1.2]{burago2001course} and \cite[Def. 2.2]{beran2023characterizingintrinsiclorentzianlength}. First, we define $\epsilon$-$\mu$-midpoints.

\begin{definition} 
Let \Xll be a \LpLS,
    a \\ \textbf{\boldmath{$\mu$}-midpoint} of two points $a\ll b \in X$ for $\mu \in (0,1)$ is a point $m\in X$ such that 
    \begin{equation*}
        \tau(a,m)=\mu \tau(a,b) \qquad\text{and}\qquad  \tau(m,b)=(1-\mu) \tau(a,b). 
    \end{equation*}
     An \textbf{\boldmath{$\epsilon$-$\mu$}-midpoint} of two points $a\ll b \in X$ for $\mu \in (0,1)$ and $\epsilon>0$  is a point $m\in X$ such that 
    \begin{equation*}
        |\tau(a,m)-\mu \tau(a,b)| <\epsilon \qquad\text{and}\qquad  |\tau(m,b) - (1-\mu) \tau(a,b)| <\epsilon.
    \end{equation*}
     An \textbf{\boldmath{$\epsilon$}-midpoint} of two points $a\ll b \in X$ is a point $m\in X$ such that 
    \begin{equation*}
        |\tau(a,m)-\frac{1}{2} \tau(a,b)| <\epsilon \qquad\text{and}\qquad  |\tau(m,b) - \frac{1}{2} \tau(a,b)| <\epsilon.
    \end{equation*}
    We note that an $\epsilon$-midpoint is just an $\epsilon$-$1/2$-midpoint and all the definitions above can be analogously defined for a metric space, replacing $\tau$ with $\dd$.
\end{definition}

\begin{definition}\label{def:epsmuCBA}
    \textbf{One-sided $\epsilon$-$\mu$ triangle curvature condition}\\
    A \LpLS \, \Xll has curvature bounded below (respectively, above) by $K \in \R$ if and only if $X$ can be covered by open neighbourhoods $U \subset X$ such that:  
    \begin{enumerate}
        \item The time separation function $\tau$ is continuous on $(U \times U) \cap \tau^{-1}([0, D_K))$ and this set is open.
        \item For any timelike (not necessarily geodesic) triangle $\triangle abc \subset U$ and any $\mu \in (0,1)$, there exists a function $f_\mu : \R_{\ge 0} \to \R_{\ge 0}$ with $\lim_{\epsilon \to 0} f_\mu(\epsilon) = 0$ such that the following holds.
        Let $a$ and $b$ be any two of the vertices of $\triangle abc$; then for all
        $\epsilon$-$\mu$ midpoints $m_{\epsilon\mu}$ of $a$ and $b$ which are also timelike related to the remaining vertex $c$, the inequalities 
        \begin{align*}
            &\tau(c, m_{\epsilon\mu}) \ge \tau(\bar c, \bar m_\mu) - f_\mu(\epsilon) \quad \text{and} \quad
            \tau(m_{\epsilon\mu}, c) \ge \tau(\bar m_\mu, \bar c) - f_\mu(\epsilon), \quad  \\
            &\left( \tau(c, m_{\epsilon\mu}) \le \tau(\bar c, \bar m_\mu) + f_\mu(\epsilon) \quad \text{and} \quad
              \tau(m_{\epsilon\mu}, c) \le \tau(\bar m_\mu, \bar c) + f_\mu(\epsilon)\right), 
        \end{align*}
        hold, where $\bar m_\mu$ is the $\mu$-midpoint in the corresponding comparison space.
    \end{enumerate}
\end{definition}

\begin{lemma}
\label{lem:curvature_equivalence}
The following conditions are equivalent for a strongly causal and regular \LpLS \,  \Xll:
\begin{enumerate}
    \item \Xll has a timelike upper curvature bound $K$ in the sense of the triangle condition (Definition~\ref{def:oneSidedLorTriangle}).\label{enum:curvature1}
    \item \Xll has a timelike upper curvature bound $K$ in the sense of the four-point condition (Definition~\ref{def:LorfourPts}), and locally every pair of timelike-related points can be connected by a distance realizer.\label{enum:curvature2}
    \item \Xll has a timelike upper curvature bound $K$ in the sense of the one-sided $\epsilon$-$\mu$ triangle condition (Definition~\ref{def:epsmuCBA}), and locally every pair of timelike-related points can be connected by a distance realizer.\label{enum:curvature3}
\end{enumerate}
\end{lemma}

\begin{proof}
\textbf{\ref{enum:curvature1} \(\Rightarrow\) \ref{enum:curvature2}:} This implication has been proven in Proposition 4.4 of \cite{beran2025reshetnyakmajorisationdiscreteupper}.

\textbf{\ref{enum:curvature2} $\Rightarrow$ \ref{enum:curvature3}:} 
We distinguish two cases. In the first, we assume that the $\epsilon$-$\mu$ midpoint lies almost on one of the two shorter sides. However, these are equivalent by reversing the time orientation, so without loss of generality, we assume that it lies almost on the longer of the two shortest sides.\\
Given $a \ll b \ll c \in U$ in some four-point curvature neighbourhood $U \subset X$, fix $\mu \in (0,1)$ and consider $\epsilon$-$\mu$ midpoints $m_{\epsilon\mu}$ between $b$ and $c$.
Consider the two triangles $\triangle a b m_{\epsilon\mu}$ and $\triangle b m_{\epsilon\mu} c$ and construct their comparison triangles $\triangle \bar a \bar b \bar m_{\epsilon\mu}$ and $\triangle \bar b \bar m_{\epsilon\mu} \bar c \subset \mathbb{L}^2_K$ such that they are glued along their common edge $\bar b \bar m_{\epsilon\mu}$ in the same direction.

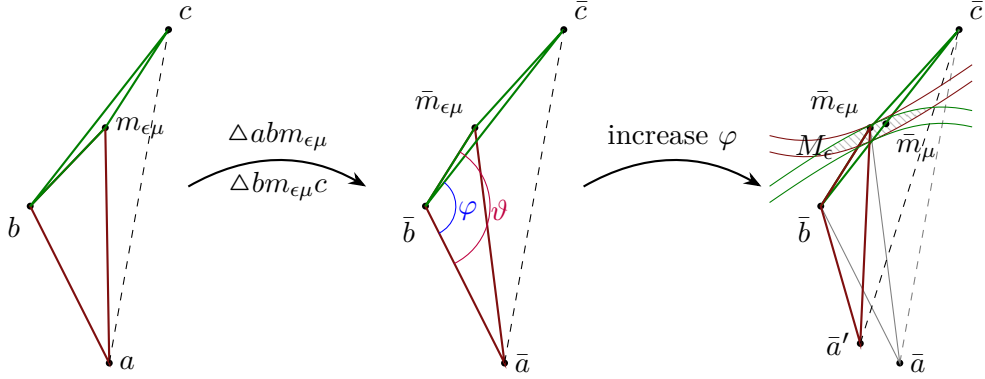
\begin{center}
\begin{tikzpicture}[scale=2.6, >=Stealth]
\coordinate (a) at (0, 0);
\coordinate (b) at (-0.4, 0.8);
\coordinate (m) at (-0.02, 1.2); 
\coordinate (c) at (0.3, 1.7);

\fill (a) circle (0.5pt) node[right] {$a$};
\fill (b) circle (0.5pt) node[below left] {$b$};
\fill (m) circle (0.5pt) node[right] {$m_{\epsilon\mu}$};
\fill (c) circle (0.5pt) node[above right] {$c$};

\draw[thick,red] (a) -- (b) -- (m) -- cycle; 
\draw[thick, green] (b) -- (m) -- (c) -- cycle;          
\draw [dashed] (a) -- (c);

\coordinate (A) at (0.4,0.9);
\coordinate (B) at (1.3,0.9);
\draw[->, thick, bend left] (A) to node[midway, above] {$\triangle a b m_{\epsilon\mu}$} node[midway, below] {$\triangle b m_{\epsilon\mu} c$} (B);

\coordinate (ab) at (2, 0);
\coordinate (bb) at (1.6, 0.8);
\coordinate (mb) at (1.85, 1.2); 
\coordinate (cb) at (2.3, 1.7);

\fill (ab) circle (0.5pt) node[right] {$\bar a$};
\fill (bb) circle (0.5pt) node[below left] {$\bar b$};
\fill (mb) circle (0.5pt) node[above left] {$\bar m_{\epsilon\mu}$};
\fill (cb) circle (0.5pt) node[above right] {$\bar c$};

\draw[thick, red] (ab) -- (bb) -- (mb) -- cycle; 
\draw[thick, green] (bb) -- (mb) -- (cb) -- cycle;          
\draw [dashed] (ab) -- (cb);

\draw[blue] pic["$\varphi$", draw=black,blue, angle radius=0.4cm, angle eccentricity=1.4] {angle = ab--bb--cb};
\draw[purple] pic["$\vartheta$", draw=black,purple, angle radius=0.85cm, angle eccentricity=1.15] {angle = ab--bb--mb};

\coordinate (A2) at (2.4,0.9);
\coordinate (B2) at (3.3,0.9);
\draw[->, thick, bend left] (A2) to node[midway, above] {increase $\varphi$} (B2);

\coordinate (abp) at (4, 0);
\coordinate (abp2) at (3.8, 0.1);
\coordinate (bbp) at (3.6, 0.8);
\coordinate (mbp) at (3.85, 1.2); 
\coordinate (mbp2) at (3.93, 1.22); 
\coordinate (cbp) at (4.3, 1.7);

\fill (abp) circle (0.5pt) node[right] {$\bar a$};
\fill (abp2) circle (0.5pt) node[left] {$\bar a'$};
\fill (bbp) circle (0.5pt) node[below left] {$\bar b$};
\fill (mbp) circle (0.5pt) node[above left] {$\bar m_{\epsilon\mu}$};
\fill (mbp2) circle (0.5pt) node[below right] {$\bar m_{\mu}$};
\fill (cbp) circle (0.5pt) node[above right] {$\bar c$};

\draw[gray] (abp) -- (bbp) -- (mbp) -- cycle; 
\draw[thick, green] (bbp) -- (mbp) -- (cbp) -- cycle;          
\draw [dashed,gray] (abp) -- (cbp);
\draw[thick, red] (abp2) -- (bbp) -- (mbp) -- cycle; 
\draw [dashed] (abp2) -- (cbp);
\begin{axis}[axis lines=none,width=80pt, at={(600,153)}]
\addplot[name path=f1,red,domain={-1:1}] {cosh(arcsinh(x+0.65))+0.1};
\addplot[name path=f2,red,domain={-1:1}] {cosh(arcsinh(x+0.65))-0.1};
\addplot[name path=f3,green,domain={-1:1}] {-cosh(arcsinh(x-0.7))+2.51};
\addplot[name path=f4,green,domain={-1:1}] {-cosh(arcsinh(x-0.7))+2.31};

\addplot[pattern=north west lines, pattern color=gray!50]fill between[of=f3 and f2, soft clip={domain=-0.55:0}];
\addplot[pattern=north west lines, pattern color=gray!50]fill between[of=f3 and f2, soft clip={domain=0:0.6}];
\end{axis}
\draw (3.57,1.1) node {$M_\epsilon$};

\end{tikzpicture}
\end{center}

The four-point condition states that, in this scenario $\tau(a,c) \leq \tau(\bar a,\bar c)$.
In a second step, we move the point $\bar a$ to $\bar a'$ in $\mathbb{L}^2_K$ while fixing its distance to $\bar b$ such that $\tau(a,c) = \tau(\bar a',\bar c) \leq \tau(\bar a,\bar c)$.
To see that $\tau(\bar a',\bar m_{\epsilon\mu}) \geq \tau(\bar a,\bar m_{\epsilon\mu})$, we use twice the monotonicity of a hinge:\\
Consider the hinge $\bar a \bar b \bar c$ with the angle $\varphi$ at $\bar b$. The monotonicity of the Lorentzian law of cosines tells us that the side length $\tau(\bar a, \bar c)$ is monotonically decreasing in $\varphi$. 
Since we want to decrease the distance $\tau(\bar a, \bar c)$, $\varphi$ must increase to obtain $\bar a'$.
Since $\varphi$ and $\vartheta$ lie in the same plane and the triangle $\triangle \bar b \bar m_{\epsilon\mu} \bar c$ remains fixed, the difference $\vartheta - \varphi$ stays constant. Hence, $\vartheta$ decreases, which, by the Lorentzian law of cosines, indicates that $\tau(\bar a, \bar m_{\epsilon\mu})$ decreases. \\
In conclusion, we have found four comparison points $\bar a', \bar b, \bar c, \bar m_{\epsilon\mu}$ for $a, b, c, m_{\epsilon\mu}$ where all side lengths agree, except for $\tau(a, m_{\epsilon\mu}) \geq \tau(\bar a', \bar m_{\epsilon\mu})$. Thus, $\triangle \bar a' \bar b \bar c$ is a comparison triangle in $\mathbb{L}^2_K$ and $\bar m_{\epsilon\mu}$ is an $\epsilon$-$\mu$ midpoint of $\bar b \bar c$. \\
Let $\bar m_\mu$ be the actual $\mu$-midpoint of $\bar b \bar c$. For a fixed triangle $\triangle \bar a' \bar b \bar c \in \mathbb{L}^2_K$ and given $\mu \in (0,1)$, the positions of all possible $\epsilon$-$\mu$ midpoints of $\bar b \bar c$ form an $\epsilon$-dependent neighbourhood $M_\epsilon$ (bounded by four hyperboloids) around $\bar m_\mu$ which converges to $\{\bar m_\mu\}$ as $\epsilon \to 0$.\\
Hence, every point $m_{\epsilon\mu} \in M_\epsilon$ satisfies
\[
|\tau(\bar a', m_{\epsilon\mu}) - \tau(\bar a', \bar m_{\mu})| \leq f(\epsilon)
\]
for some $f : \R_{\geq 0} \to \R_{\geq 0}$ with $\lim_{\epsilon \to 0} f(\epsilon) = 0$. This implies
\[
\tau(a, m_{\epsilon\mu}) \geq \tau(\bar a, \bar m_{\mu}) - f(\epsilon).
\]

For the second case, consider the same timelike triangle $\triangle a b c \in U$ and an $\epsilon$-$\mu$ midpoint $m_{\epsilon\mu}$ between $a$ and $c$. Consider these four points as two triangles $\triangle a b m_{\epsilon\mu}$ and $\triangle b m_{\epsilon\mu} c$, construct two comparison triangles in $\mathbb{L}^2_K$, and glue them along their common edge $\gamma_{\bar b \bar m_{\epsilon\mu}}$ in opposite directions. The four-point condition states that $\tau(b, m_{\epsilon\mu}) \geq \tau(\bar b, \bar m_{\epsilon\mu})$. 
Let $\bar m_\mu$ be the actual $\mu$-midpoint of $\bar a \bar c$. Then, by the same argument as above, we can conclude that there exists $f(\epsilon)$ with $\lim_{\epsilon \to 0} f(\epsilon) = 0$ such that
\[
\tau(b, m_{\epsilon\mu}) \geq \tau(\bar b, \bar m_{\mu}) - f(\epsilon).
\]

\textbf{\ref{enum:curvature3} $\Rightarrow$ \ref{enum:curvature1}:}
For a given timelike triangle $\triangle a b c \in U$ and a $\mu$-midpoint $m_\mu$ on one of the distance realizers (assume for this argument that $m_\mu$ is a $\mu$-midpoint for $b$ and $c$), we can take a sequence of $\frac{1}{n}$-$\mu$ midpoints which converges to $m_\mu$ as $n \to \infty$ (such a sequence always exists, for example by taking the sequence along the distance realizer $\gamma_{bc}$ itself). The inequality
\[
\tau(a, m_{\frac{1}{n}\mu}) \geq \tau(\bar a, \bar m_{\mu}) - f\left(\tfrac{1}{n}\right) \xrightarrow[n\to\infty]{} \tau(\bar a, \bar m_{\mu}),
\]
together with the continuity of $\tau$, this implies
\[
\tau(a, m_{\mu}) \geq \tau(\bar a, \bar m_{\mu}).
\]

\end{proof}

\begin{remark}
    Note that for \ref{enum:curvature2} $\Rightarrow$ \ref{enum:curvature3}, the local existence of geodesics is not needed. This implication holds for any \LpLS.
\end{remark}

The same equivalence of curvature conditions can also be proven in the case of curvature bounded below.
\begin{lemma}
\label{lem:curvature_equivalence_lower}
The following conditions are equivalent for a strongly causal and regular \LpLS \, \Xll:
\begin{enumerate}
    \item \Xll has a lower curvature bound $K$ in the sense of the triangle condition (Definition~\ref{def:oneSidedLorTriangle}).\label{enum:curvature12}
    \item \Xll has a lower curvature bound $K$ in the sense of the four-point condition (Definition~\ref{def:LorfourPts}), and locally every pair of timelike-related points can be connected by a distance realizer.\label{enum:curvature22}
    \item \Xll has a lower curvature bound $K$ in the sense of the one-sided $\epsilon$-$\mu$ triangle condition (Definition~\ref{def:epsmuCBA}), and locally every pair of timelike-related points can be connected by a distance realizer.\label{enum:curvature32}
\end{enumerate}
\end{lemma}

\begin{proof}
\textbf{\ref{enum:curvature12} $\Rightarrow$ \ref{enum:curvature22}:} See Theorem 5.1 of \cite{secCurvBouns}.

\textbf{\ref{enum:curvature22} $\Rightarrow$ \ref{enum:curvature32}:} The idea for the proof is the same as in \textbf{\ref{enum:curvature2} $\Rightarrow$ \ref{enum:curvature3}} of \Cref{lem:curvature_equivalence}. 
We distinguish two cases. In the first, we show the $\epsilon$-$\mu$ triangle condition for an $\epsilon$-$\mu$ midpoint on one of the shorter edges. Moreover, due to symmetry, we can restrict the discussion to the past edge (see picture on the left), and in the second case, the $\epsilon$-$\mu$ midpoint is on the longest edge (picture on the right).\\
Given $a\ll b \ll c\in U$ in a four point neighbourhood $U\subset X$ and $\mu\in(0,1)$, and let $m_{\epsilon\mu}$ be an $\epsilon$-$\mu$ midpoint  between $a$ and $b$. Now consider the two triangles $\triangle a m_{\epsilon \mu}b$ and $\triangle a m_{\epsilon \mu}c$ and construct their comparison triangles in $\mathbb{L}^2_K$, glued together along the edge $\bar a \bar m_{\epsilon \mu}$ in opposite directions. 
Since this is a valid four point comparison configuration for lower curvature bounds, we know $\tau(b,c)\geq \tau (\bar b, \bar c)$.
As in the proof of \Cref{lem:curvature_equivalence}, we can also see these four points as two hinges with the angles $\varphi = \tilde\measuredangle_{\bar a }^K \bar m_{\epsilon \mu} \bar c$ and $\vartheta = \tilde\measuredangle_{\bar a}^K \bar v \bar c$. Closing these angles and using the law of cosine twice, we get a new point $\bar c'$ such that $\tau(b, c) = \tau (\bar b,\bar c')$ and $\tau(m_{\epsilon\mu}, c) \leq \tau (\bar m_{\epsilon\mu},\bar c')$. Since $m_{\epsilon \mu }$ is an $\epsilon$-$\mu$ midpoint of $a$ and $b$, $\bar m_{\epsilon \mu }$ will be in an $\epsilon$-dependent neighbourhood of $\bar m_{\mu}$ which converges to a point for $\epsilon\to0$. Therefore, $|\tau(\bar m_{\epsilon \mu }, \bar c') -\tau(\bar m_{\mu }, \bar c') |< f(\epsilon)$ with $\lim_{\epsilon\to 0}f(\epsilon) =0$. Altogether, we get that 
$$\tau(m_{ \epsilon \mu }, c) \leq \tau(\bar m_{ \mu }, \bar c')+ f(\epsilon).$$
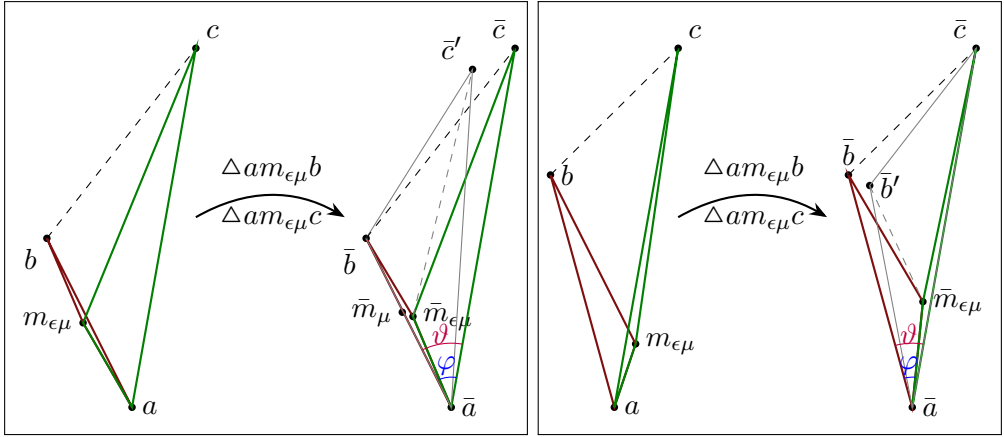
\begin{center}
\fbox{\begin{tikzpicture}[scale=2.8, >=Stealth]
\coordinate (a) at (0, 0);
\coordinate (b) at (-0.4, 0.8);
\coordinate (m) at (-0.23, 0.4); 
\coordinate (c) at (0.3, 1.7);

\fill (a) circle (0.5pt) node[right] {$a$};
\fill (b) circle (0.5pt) node[below left] {$b$};
\fill (m) circle (0.5pt) node[left] {$m_{\epsilon\mu}$};
\fill (c) circle (0.5pt) node[above right] {$c$};

\draw[thick,red] (a) -- (b) -- (m) -- cycle; 
\draw[thick, green] (a) -- (m) -- (c) -- cycle;          
\draw [dashed] (b) -- (c);

\coordinate (A) at (0.3,0.9);
\coordinate (B) at (1,0.9);
\draw[->, thick, bend left] (A) to node[midway, above] {$\triangle a m_{\epsilon\mu} b$ } node[midway, below] {$\triangle a m_{\epsilon\mu} c$} (B);

\coordinate (ab) at (1.5, 0);
\coordinate (bb) at (1.1, 0.8);
\coordinate (mb) at (1.32, 0.43); 
\coordinate (mbp) at (1.27, 0.45); 
\coordinate (cb) at (1.8, 1.7);
\coordinate (cbp) at (1.6, 1.6);

\fill (ab) circle (0.5pt) node[right] {$\bar a$};
\fill (bb) circle (0.5pt) node[below left] {$\bar b$};
\fill (mb) circle (0.5pt) node[right] {$\bar m_{\epsilon\mu}$};
\fill (mbp) circle (0.5pt) node[left] {$\bar m_{\mu}$};
\fill (cb) circle (0.5pt) node[above left] {$\bar c$};
\fill (cbp) circle (0.5pt) node[above left] {$\bar c'$};

\draw[thick, red] (ab) -- (bb) -- (mb) -- cycle; 
\draw[thick, green] (ab) -- (mb) -- (cb) -- cycle;          
\draw[gray] (ab) -- (bb) -- (cbp) -- cycle;          
\draw [dashed] (bb) -- (cb);
\draw [dashed, gray] (mb) -- (cbp);

\draw[blue] pic["$\varphi$", draw=black,blue, angle radius=0.4cm, angle eccentricity=1.4] {angle = cb--ab--mb};
\draw[purple] pic["$\vartheta$", draw=black,purple, angle radius=0.85cm, angle eccentricity=1.15] {angle = cb--ab--bb};
\end{tikzpicture}} 
\fbox{\begin{tikzpicture}[scale=2.8, >=Stealth]
\coordinate (a) at (0, 0);
\coordinate (b) at (-0.3, 1.1);
\coordinate (m) at (0.1, 0.3); 
\coordinate (c) at (0.3, 1.7);

\fill (a) circle (0.5pt) node[right] {$a$};
\fill (b) circle (0.5pt) node[right] {$b$};
\fill (m) circle (0.5pt) node[right] {$m_{\epsilon\mu}$};
\fill (c) circle (0.5pt) node[above right] {$c$};

\draw[thick,red] (a) -- (m) -- (b) -- cycle; 
\draw[thick, green] (a) -- (m) -- (c) -- cycle;          
\draw [dashed] (b) -- (c);

\coordinate (A) at (0.3,0.9);
\coordinate (B) at (1,0.9);
\draw[->, thick, bend left] (A) to node[midway, above] {$\triangle a m_{\epsilon\mu} b$ } node[midway, below] {$\triangle a m_{\epsilon\mu} c$} (B);

\coordinate (ab) at (1.4, 0);
\coordinate (bb) at (1.1, 1.1);
\coordinate (bbp) at (1.2, 1.05);
\coordinate (mb) at (1.45, 0.5); 
\coordinate (cb) at (1.7, 1.7);

\fill (ab) circle (0.5pt) node[right] {$\bar a$};
\fill (bb) circle (0.5pt) node[above] {$\bar b$};
\fill (bbp) circle (0.5pt) node[right] {$\bar b'$};
\fill (mb) circle (0.5pt) node[right] {$\bar m_{\epsilon\mu}$};
\fill (cb) circle (0.5pt) node[above left] {$\bar c$};

\draw[thick, red] (ab) -- (bb) -- (mb) -- cycle; 
\draw[thick, green] (ab) -- (mb) -- (cb) -- cycle;          
\draw [dashed] (bb) -- (cb);
\draw[gray] (ab) -- (bbp) -- (cb) -- cycle;          
\draw [dashed,gray] (bbp) -- (mb);

\draw[blue] pic["$\varphi$", draw=black,blue, angle radius=0.4cm, angle eccentricity=1.4] {angle = mb--ab--bb};
\draw[purple] pic["$\vartheta$", draw=black,purple, angle radius=0.85cm, angle eccentricity=1.15] {angle = cb--ab--bb};
\end{tikzpicture}}
\end{center}
In the second case, we assume the $\epsilon$-$\mu$ midpoint to be on the edge $ac$, in the past of $b$. We build comparison triangles for $\triangle a m_{\epsilon\mu} b$ and $\triangle a m_{\epsilon\mu} c$ and glue them together in opposite directions. The four point condition states that $\tau(b,c)\geq \tau( \bar b, \bar c)$. Keeping all the side lengths except $\tau(\bar m_{\epsilon\mu},\bar b)$ and $\tau(\bar b,\bar c)$ fixed, we can move $\bar b$ to $\bar b'$ such that
$$\tau( b, c) = \tau(\bar b',\bar c)\quad \text{and}\quad\tau(m_{\epsilon\mu},b) \leq \tau(\bar m_{\epsilon\mu},\bar b') \leq \tau(\bar m_{\mu},\bar b') + f(\epsilon).$$ 

\textbf{\ref{enum:curvature32} $\Rightarrow$ \ref{enum:curvature12}:}  
Analogous to \textbf{\ref{enum:curvature3} $\Rightarrow$ \ref{enum:curvature1}:} of \Cref{lem:curvature_equivalence}.
\end{proof}

\section{The Space of Directions}
\subsection{Definitions}
In the following, we restrict our attention to future-directed, timelike curves starting from a common point $p\in X$ and use the unsigned version of the angles (that is, all angles are understood to be nonnegative). A more detailed presentation of the following definitions and lemmas can be found in Sections 3.2 and subsequent sections of \cite{HyperbolicAnglesBeran_2023}, as well as in \cite{alexander2021generalizedconeslorentzianlength}.\\
We start by defining the space of directions. 
\begin{definition}
    Let \Xll be a \LpLS \, with $\tau$ locally finite, and let $p \in X$. We define the set
    \[
        D_p^+ \mDef \Bigl\{\gamma : [0,\epsilon) \to X \,\Big|\,
        \gamma \text{ future-directed, timelike, geodesic, and } \gamma(0) = p \Bigr\}
    \]
    to be the set of directions of future-directed timelike geodesics starting at $p$. We equip this set with an equivalence relation $\sim$ by declaring $\gamma_1 \sim \gamma_2$ if $\measuredangle_p(\gamma_1, \gamma_2) = 0$. \\
    We denote the quotient set by
    \[
        \mathcal{D}_p^+ \mDef D_p^+ / \sim.
    \]
    In fact, this space is a metric space with metric $\dd_D(\gamma_1, \gamma_2) = \measuredangle_p(\gamma_1, \gamma_2)$.
    Furthermore, we call the metric completion of $\bigl(\mathcal{D}_p^+,\measuredangle_p\bigr)$, denoted by
    \[
        \bigl(\Sigma_p^+,\dd_\Sigma = \measuredangle_p \bigr),
    \]
    the \textbf{space of future directions} at $p$.
\end{definition}
\begin{example}
If we have a smooth $n$-dimensional Riemannian manifold, the space of directions is isometric to the sphere $S^{n-1}$, which has constant sectional curvature of $1$.\\
For Minkowski space (or any smooth Lorentzian manifold), the space of timelike, future pointing directions is the hyperbolic space $\mathbb{H}^{n-1}\subset \R^{1,n-1}$, which has a sectional curvature of $-1$ \cite[Prop. 29]{ONeill1983}. Therefore, it is natural to expect sectional curvature bounds $-1$ for the space of directions in Lorentzian length spaces.
\end{example}
To take into account the speed of geodesics, we construct a cone over the space of directions to obtain an analogous notion of the tangent space. In the Lorentzian setting, we need the definition of the Minkowski cone for this:
\begin{definition} \label{def:cone}
     Let $(Y,\dd_Y)$ be a metric space. We define the \textbf{Minkowski cone} \((X,\dd_X,\ll,\leq,\tau)\) over $Y$ to be the set
     \[
      X \coloneq Con(Y) = ([0,\infty)\times Y)/(\{0\} \times Y) \text{ with } x_i=(r_i,y_i)\in X \, ,
     \]
     equipped with the standard cone metric, 
     \[
     \dd_X(x_1,x_2) \coloneq \sqrt{r_1^2 + r_2^2 - 2r_1 r_2 \cos(\min\{\pi,d_Y(y_1,y_2)\})} \, .
     \]
     Moreover, the causal relations are defined to be
     \[
     x_1 \leq x_2 \iff r_1^2 + r_2^2 - 2r_1 r_2 \cosh(d_Y(y_1,y_2)) \geq 0 \, ,
     \]
     with time separation function,
     \[
     \tau(x_1,x_2) \coloneq \begin{cases} \sqrt{r_1^2 + r_2^2 - 2r_1 r_2 \cosh(d_Y(y_1,y_2))} &\text{if} \quad x_1 \leq x_2, \\ 0 & \text{otherwise,}\end{cases}\\
     \]
     inducing the timelike relation
     \[
     x_1 \ll x_2 \iff \tau(x_1,x_2)>0 \, .
     \]
\end{definition}
\begin{lemma}
Let \((Y, \dd_Y)\) be a metric space and let \(X = \mathrm{Con}(Y)\) be equipped with the cone metric \(d_c\). Then \((X, d_c, \ll, \leq, \tau)\) is a Lorentzian pre-length space with time separation function \(\tau\) defined as above.
\end{lemma}
\begin{proof}
See Proposition 2.2 of \cite{alexander2021generalizedconeslorentzianlength}.
\end{proof}

\begin{definition}
Let \Xll be a \LpLS \, with \(\tau\) locally finite and let \(p \in X\). The \textbf{future tangent cone} at \(p\) is the Lorentzian cone over the space of future directions at \(p\), defined by
\[
T_p^+ \mDef \mathrm{Con}(\Sigma_p^+).
\]
Similarly, 
\[
\mathcal{T}_p \mDef \mathrm{Con}(\mathcal{D}_p^+)
\]
is the cone over the set of directions before completion, and is dense in \(T_p^+\).
\end{definition}
The definition of the tangent cone naturally gives rise to the definition of the logarithm and exponential map by identifying points in the \LpLS \, and the tangent cone.
\begin{definition}[Logarithmic Map]\label{def-log}
Let \Xll be a locally (in the neighbourhood $U\subset X$) uniquely timelike geodesic Lorentzian length space and $p\in U$. For any \(x \in U\) with \(p \ll x\) (respectively \(x \ll p\)), there exists a unique future-directed (respectively past-directed) timelike geodesic \(\gamma\) from \(p\) to \(x\). The \textbf{future} and \textbf{past logarithmic maps} are defined by
\[
\log_p^+(x) \mDef \bigl(\tau(p,x), [\gamma]\bigr) \in T_p^+\, .
\]
Thus, 
\[
\log_p^+ \colon U \cap I^+(p) \to T_p^+
\]
maps points to elements of the future or past tangent cone.
\end{definition}

\begin{definition}[Exponential Map]\label{def:exp}
Let \Xll be a locally uniquely timelike geodesic Lorentzian length space, let \(p \in X\). The \textbf{future exponential map}
\[
\exp_p^+ \colon \log_p^+\bigl(U \cap I^+(p)\bigr) \subset T_p^+ \to U
\]
is defined by
\[
\exp_p^+\bigl(r,[\gamma]\bigr) \mDef \tilde\gamma(r),
\]
where \(\tilde\gamma \in [\gamma]\) denotes the unique representative parametrized by \(\tau\)-arclength.
\end{definition}
\begin{remark}
    The exponential map is only well defined if geodesics do not branch. Otherwise, there could be multiple arclength parametrized curves with the same starting direction. Therefore, we consider the exponential function as a set valued map.
\end{remark}

\subsection{Under Upper Curvature Bounds the Space of Directions is a Length Space}
\begin{theorem} \label{them:SoD_lengthspace}
Let \Xll be a Lorentzian length space with timelike sectional curvature bounded above by \(K \in \mathbb{R}\).
Then the space of directions \(\Sigma_p^+\) at any point \(p \in X\) is a length space, and the tangent cone is a Lorentzian pre-length space.
\end{theorem}

\begin{proof}
By \cite{burago2001course} (Theorem 2.4.16) it is enough to show that every two points in $\Sigma^+_p$ have an $\epsilon$-midpoint for all $\epsilon >0$. \\
Since $\Sigma^+_p$ is the completion of the space of all future-directed timelike distance realizers (quotiented by $\measuredangle_p=0$), it is sufficient to check the condition for two points in $\Sigma^+_p$, whose distance is achieved by distance realizers $\alpha,\gamma :[0,1]\to M$ with $\measuredangle_p(\alpha ,\gamma)<\infty$.\\
Since $\lim_{t,s\to 0} \Tilde{\measuredangle}^K_p\alpha(t)\gamma(s) = \measuredangle_p(\alpha,\gamma)$, we can choose $t,s>0$ such that \\
$p\ll a = \alpha(t) \ll c = \gamma(s)$ and
\begin{equation} \label{eq:epsilonalphagamma}
    |\Tilde{\measuredangle}^K_p\alpha(t)\gamma(s) -\measuredangle_p(\alpha,\gamma)| <\epsilon
    \end{equation}
for any given $\epsilon>0$. Without loss of generality, we can also assume that $a$ and $c$ are in a curvature comparison neighbourhood of $p$.\\
The comparison angle $\Tilde{\measuredangle}^K_p$ is a continuous function, depending on the three time separations of a given triangle. 
Since we are in a curvature comparison neighbourhood, $\tau$ is continuous. By the intermediate value theorem, we can find a point on the distance realizer between $a$ and $c$, $b\in \gamma_{ac}$ such that \begin{equation} \label{angleequal}
    \Tilde{\measuredangle}^K_p a b = \Tilde{\measuredangle}^K_p b c\, .
\end{equation}
We now construct two comparison triangles $\triangle \bar{p}\bar{a}\bar{b}$ and $\triangle \bar{p}\bar{b}\bar{c}$ in $\mathbb{L}^2(K)$ with common edge $\bar{p}\bar{b}$ (with the two triangles being on opposite sides of this edge).
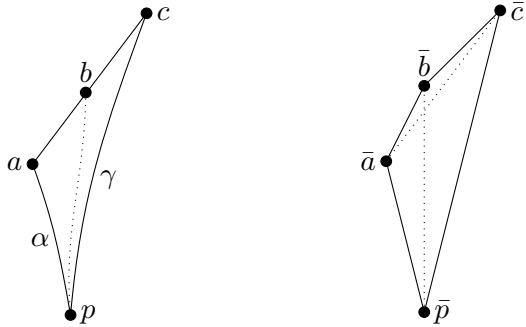
\begin{center}
\begin{tikzpicture}
    \filldraw (0,0) circle (2pt) node[right] {$p$};
    \filldraw (-0.5,2) circle (2pt) node[left] {$a$};
    \filldraw (0.2,2.95) circle (2pt) node[above] {$b$};
    \filldraw (1,4) circle (2pt) node[right] {$c$};
    \draw (0,0) to[out=100, in=290] (-0.5,2);
    \draw (-0.4,1) node {$\alpha$};
    \draw (-0.5,2) -- (1,4);
    \draw (0,0) to[out=85,in=250] (1,4);
    \draw (0.5,1.8) node {$\gamma$};
    \draw[dotted] (0,0) to[out=95,in=270] (0.2,3);
\end{tikzpicture}
\hspace{2 cm}
\begin{tikzpicture}
    \filldraw (0,0) circle (2pt) node[right] {$\bar{p}$};
    \filldraw (-0.5,2) circle (2pt) node[left] {$\bar{a}$};
    \filldraw (0,3) circle (2pt) node[above] {$\bar{b}$};
    \filldraw (1,4) circle (2pt) node[right] {$\bar{c}$};
    \draw (0,0) --(-0.5,2);
    \draw (-0.5,2) -- (0,3);
    \draw (0,3) -- (1,4);
    \draw (0,0) -- (1,4);
    \draw[dotted] (0,0) -- (0,3);
    \draw[dotted] (-0.5,2) -- (1,4);
\end{tikzpicture}
\end{center}
Due to the inverse triangle inequality, we have
\begin{equation*}
    \tau(a,c) = \tau(a,b) + \tau(b,c) = \bar{\tau}(\bar{a},\bar{b}) + \bar{\tau}(\bar{b},\bar{c}) \leq \bar{\tau}(\bar{a},\bar{c}).
\end{equation*}
Using that the comparison angle (the non signed version) is monotonically decreasing in dependence on the time separation of the side length opposite to the angle, it follows that

\begin{equation} \label{eq:angleaccomp}
    \tilde\measuredangle_{\bar{p}}^K \bar{a} \bar{c}
    \leq \tilde\measuredangle_{p}^K ac.
\end{equation}
From this construction, we obtain the following inequality: 
\begin{equation*}
      2 \measuredangle_p (\alpha, \beta) \leq 
      2\tilde\measuredangle_{\bar{p}}^K \bar{a} \bar{b} = \tilde\measuredangle_{\bar{p}}^K \bar{a} \bar{b} + \tilde\measuredangle_{\bar{p}}^K \bar{b} \bar{c} 
      = \tilde\measuredangle_{\bar{p}}^K \bar{a} \bar{c}
    \leq \tilde\measuredangle_{p}^K ac
    \leq \measuredangle_{p} (\alpha, \gamma) + \epsilon.
\end{equation*}
Here, we first applied \Cref{rem:anglecomp}, then equation (\ref{angleequal}), then used that angles sum in a comparison neighbourhood, then applied equation (\ref{eq:angleaccomp}) and finally equation (\ref{eq:epsilonalphagamma}). Similarly for $\measuredangle_p (\beta ,\gamma)$,  

\begin{equation}
    \measuredangle_p (\alpha, \beta) \leq \frac{1}{2} \measuredangle_p (\alpha, \gamma) + \frac{1}{2} \epsilon,  \qquad \qquad \measuredangle_p  (\beta, \gamma) \leq \frac{1}{2} \measuredangle_p (\alpha, \gamma) + \frac{1}{2} \epsilon.
\end{equation}
Combining this with $\measuredangle_p (\alpha, \gamma) \leq \measuredangle (\alpha, \beta) + \measuredangle_p (\beta, \gamma)$ obtains $\varepsilon-$midpoints,
\begin{equation*}
    |\measuredangle_p (\alpha,\beta) -\frac{\measuredangle_p (\alpha,\gamma) }{2}| \leq \epsilon, \qquad \qquad  |\measuredangle_p (\beta,\gamma) -\frac{\measuredangle_p (\alpha,\gamma) }{2}| \leq \epsilon.
\end{equation*}
To see that the tangent cone is a \LpLS, we apply Proposition 2.2 of \cite{alexander2021generalizedconeslorentzianlength}.
\end{proof}

\section{Upper Curvature Bounds for the Space of Directions}
\subsection{Curvature Bounds of the Tangent Cone}

\begin{lemma} \label{lem:a_0 ll a_1}
    Let \Xll be a \LpLS \, satisfying the conditions of \Cref{def-log} and having timelike sectional curvature bounded below by $K \in \mathbb{R}$. Let $T_p^+$ be the tangent cone at a point $p \in X$. For points $\bar{a}_0,\bar{a}_1\in \log^+_p(U\cap I^+(p)) \subset T_p^+$ that are sufficiently close to the origin $O$ of $T_p^+$ and satisfying $\bar{a}_0 \ll \bar{a}_1$, we have $\exp_{p}(\bar{a}_0) \ll \exp_{p}(\bar{a}_1)$.
\end{lemma}

\begin{proof}    
    We will show for all timelike related points that are sufficiently close to the origin in \( T_p^+ \), their images under the exponential map are timelike related.
    Let \( T > 0 \) and \( \alpha, \beta: [0,T] \to X \) be two timelike geodesics parametrized by \( \tau \)-arclength such that \( \alpha(0) = \beta(0) = x \). Let \( \omega = \measuredangle_p(\alpha, \beta) \). For all \( 0 < t < T \), we define \( \alpha_t: [0,t] \to X \) to be the curve \( \alpha \) restricted to the smaller domain and define \( \beta_t \) similarly. Mapping these geodesics under the logarithm map yields points $\log_p(\alpha_t) = (t, [\alpha])$, $\log_p(\beta_t) = (t, [\beta]) \in T_p^+$. Let \( 0 < \mu < 1 \) and suppose \( (\mu t, [\alpha]) \ll (t, [\beta]). \) By the definition of \( \ll \) on \( T_p^+ \), this is equivalent to requiring:
    \[
        \bar{\tau} \big((\mu t, [\alpha]), (t, [\beta])\big) = \sqrt{\mu t^2 + t^2 - 2 \mu t^2 \cosh(\omega)} > 0,
    \]
    and \( 0 < \mu < 1 \). Hence, for any positive \( t \), $\alpha_{\mu t} \ll \beta_t$ if and only if $\mu < e^{-\omega}$.
    Since $\exp_p((\mu t, [\alpha])) = \alpha(\mu t)$ and  $\exp_p((t, [\beta])) = \beta(t)$, we may apply Proposition 3.2 from \cite{HyperbolicAnglesBeran_2023} and conclude that for all $ 0 < \mu < e^{-\omega} $ and sufficiently small  $t$, $\alpha(\mu t) \ll \beta(t)$.
\end{proof}

\begin{lemma}
\label{lem-distancesInConeConverge}
The exponential map is an almost isometry for points close to the origin.\\
More precisely, under the conditions of \Cref{def-log}, for any two timelike related points $\bar{a}_i = (r_i,[\gamma_i]) \ll \bar{a}_j = (r_j,[\gamma_j])\in T^+_p$, the points $a_i^\lambda := \exp_p (\lambda r_i, [\gamma_i])\in U\subset X$ satisfy
$$\lim_{\lambda\to 0}\lambda^{-1}\tau(a^\lambda_i, a^\lambda_j) = \bar{\tau}(\bar{a}_i,\bar{a}_j).$$
\end{lemma}
\begin{proof} We first check that the limit exists. For points $\bar{a}_i, \ll \bar{a}_j$, the scaled down points are also timelike related $\bar{a}_i^\lambda = (\lambda r_i, [\gamma_i]) \ll \bar{a}_j^\lambda = (\lambda r_j, [\gamma_j])$ and hence, by \Cref{lem:a_0 ll a_1}, for all sufficiently small $\lambda$, $a_i^\lambda \ll a_j^\lambda$. Therefore, $\tau(a_i^\lambda,a_j^\lambda) > 0$ for sufficiently small $\lambda$. Moreover, by \Cref{rem:comparisonAngle}, we have that  
\[
\lim_{\lambda \to 0} \tilde{\measuredangle}^{0}_p \left(\gamma_i(\lambda r_i ), \gamma_j(\lambda r_j)\right) = \arcosh \left( \frac{\tau(p, a_i^{\lambda})^2 + \tau(p, a_j^{\lambda})^2 - \tau(a_i^{\lambda}, a_j^{\lambda})^2}{2\tau(p, a_i^{\lambda}) \tau(p, a_j^{\lambda})} \right).
\]
The proof is a simple application of \Cref{def:cone}, the previous equation and using the $K=0$ comparison angle formula (since the angle exists we can use any $K$ since they all agree. Therefore, we will use $K=0$ for ease of calculation).
    \begin{align*}
    \bar{\tau} (\bar{a}_i, \bar{a}_j) &= \sqrt{r_i^2 + r_j^2 - 2 r_i r_j \cosh \big( \measuredangle_p (\gamma_i, \gamma_j) \big)} \\
    &= \sqrt{r_i^2 + r_j^2 - 2 r_i r_j \cosh \left(\lim_{t,s\to 0} \tilde{\measuredangle}_p \left( \gamma_i(t),\gamma_j(s)\right)\right)} \\
    &= \sqrt{r_i^2 + r_j^2 - 2 r_i r_j \cosh \left(\lim_{\lambda \to 0} \tilde{\measuredangle}^{0}_p \left(\gamma_i(\lambda r_i ), \gamma_j(\lambda r_j)\right)\right)} \\
    &=\! \lim_{\lambda\to 0} \! \sqrt{\!r_i^2 + r_j^2 - 2 r_i r_j \cosh \left(\! \arcosh \left( \frac{\tau(p, a_i^{\lambda})^2 + \tau(p, a_j^{\lambda})^2 - \tau(a_i^{\lambda}, a_j^{\lambda})^2}{2\tau(p, a_i^{\lambda}) \tau(p, a_j^{\lambda})} \right)\right)} \\
    &= \lim_{\lambda \to 0} \sqrt{r_i^2 + r_j^2 - 2 r_i r_j   \frac{\lambda^2 r_i^2 + \lambda^2 r_j^2 - \tau(a_i^{\lambda}, a_j^{\lambda})^2}{2\lambda^2 r_i r_j} } \\
    &= \lim_{\lambda \to 0} \sqrt{\frac{1}{\lambda^2} \tau(a_i^{\lambda}, a_j^{\lambda})^2} \\
    &= \lim_{\lambda \to 0} \lambda^{-1} \tau(a_i^{\lambda}, a_j^{\lambda}).
\end{align*}
This completes the proof.
\end{proof}

\begin{lemma} \label{lem:TConeflat}
Under the conditions of \Cref{def-log} and an upper timelike sectional curvature bound $K\in\R$, the tangent cone $T_p^+$ 
has curvature bounded above in the sense of \Cref{def:LorfourPts} by zero for any $p\in X$.
\end{lemma}
\begin{proof}
Since $\mathcal{T}_p$ is dense in $T_p^+$, it is enough to check the four-point condition on $\mathcal{T}_p$.
Let $\Bar{a}_1\ll \Bar{a}_2 \ll \Bar{a}_3 \ll \Bar{a}_4$ be a timelike four-point configuration in $T_p^+$ where $\bar{a}_i = (r_i, [\gamma_i])$. Let $U \subset X$ be a curvature comparison neighbourhood and assume that $\exp_p(\bar{a}_i) = \gamma_i(r_i) = a_i \in U$. We define $a_i^\lambda = \exp_p(\lambda r_i,[\gamma_i])$ to be the rescaled points in $X$. Since $X$ has an upper curvature bound by $K$, for each $\lambda$ we may form the first of the two necessary comparison configurations $\tilde{a}^{\lambda,1}_i \in \mathbb{L}^2(K)$ of $a_i^\lambda$ such that all the distances are preserved except for $\tau(a_2^\lambda,a_3^\lambda)$. In fact, by \Cref{def:LorfourPts} we have that 
$\tau_X(a^{\lambda}_2,a^{\lambda}_3) \geq \Tilde{\tau}_K(\tilde{a}_2^{\lambda,1},\tilde{a}_3^{\lambda,1})$.\\
The comparison space can now be scaled up (together with the comparison construction $\tilde{a}_i^{\lambda,1}$) by a factor of $\frac{1}{\lambda}$, which reduces the curvature according to $\frac{1}{\lambda}\mathbb{L}^2(K) = \mathbb{L}^2(\lambda^2 K)$.

For all pairs $(i,j)\neq(2,3)$, we have due to \Cref{lem-distancesInConeConverge}
\begin{equation*}
    \Tilde{\tau}_0 (\Tilde{a}_i^1,\Tilde{a}_j^1) \hspace{-0.08cm} \xleftarrow{\lambda\to 0} \hspace{-0.08cm} \Tilde{\tau}_{\lambda^2 K} (\frac{1}{\lambda} \Tilde{a}^{\lambda,1}_i, \frac{1}{\lambda} \Tilde{a}^{\lambda,1}_j) = \frac{1}{\lambda}  \Tilde{\tau}_{K} (\Tilde{a}^{\lambda,1}_i, \Tilde{a}^{\lambda,1}_j) 
    = \frac{1}{\lambda}  \tau_{X} (a^{\lambda}_i, a^{\lambda}_j)
    \hspace{-0.08cm} \xrightarrow{\lambda\to 0} \hspace{-0.08cm} \Bar{\tau}(\Bar{a}_i, \Bar{a}_j).
\end{equation*}
This means that we get a flat comparison neighbourhood (in $\R^{1,1}$) for the original four point construction $\Bar{a}_i$. For $(i,j)=(2,3)$ we get due to the four point inequality:
\begin{equation*}
    \Tilde{\tau}_0 (\Tilde{a}_2^1,\Tilde{a}_3^1) \hspace{-0.08cm} \xleftarrow{\lambda\to 0} \hspace{-0.08cm} \Tilde{\tau}_{\lambda^2 K} (\frac{1}{\lambda} \Tilde{a}^{\lambda,1}_2, \frac{1}{\lambda} \Tilde{a}^{\lambda,1}_3) = \frac{1}{\lambda}  \Tilde{\tau}_{K} (\Tilde{a}^{\lambda,1}_2, \Tilde{a}^{\lambda,1}_3) 
    \leq \frac{1}{\lambda}  \tau_{X} (a^{\lambda}_2, a^{\lambda}_3) \hspace{-0.08cm}\xrightarrow{\lambda\to 0} \hspace{-0.08cm} \Bar{\tau}(\Bar{a}_2, \Bar{a}_3).
\end{equation*}
Hence, we can conclude that the inequality holds in the limit and we get $\Tilde{\tau}_0 (\tilde{a}^{1}_2,\Tilde{a}^{1}_3) \leq \Bar{\tau} (\Bar{a}_2,\Bar{a}_3)$, making $T_p^+$ a space of curvature bounded above by $0$. The proof for the other four point comparison configurations comparing $\Tilde \tau_K(\tilde{a}^{\lambda,2}_1,\tilde{a}^{\lambda,2}_4)$ works analogously. 
\end{proof}

\subsection{Curvature bounds for the space of direction}

\begin{lemma} \label{lem:Ycba-1}
Let $(Y,\dd_Y)$ be a length space with a complete metric and $X=Con(Y)$ the Minkowski cone over $Y$. If $X$ is $\mu$-$\epsilon$ curvature bounded above by $0$, then $Y$ is a geodesic space of curvature bounded above by $-1$.
\end{lemma}
\begin{proof}
We first prove that $Y$ has curvature bounded above by $-1$.
For this, take three points $y_1,y_2,y_3\in Y$ and some $\epsilon$ midpoint $y_{m\epsilon}$ between $y_2$ and $y_3$.\\
Let $\bar{y}_i$ be a comparison triangle in $\mathbb{H}^2\subset \R^{1,2}$ and $\bar{y}_m$ the (real) midpoint of $\bar y_2$ and $\bar y_3$.\\
We now choose three radii $r_i>0$ such that $\bar x_i \mDef r_i \cdot \bar y_i$ satisfy $\bar x_1 \ll \bar x_2 \ll \bar x_3$ and set 
\begin{equation*}
r_m \mDef \frac{2 r_2 r_3 \cosh \frac{\dd(y_2,y_3)}{2}}{r_2+r_3}\,.    
\end{equation*}
With this setup, $\bar x_m = r_m \cdot \bar y_m$ becomes a $\mu=\frac{r_2}{r_2+r_3}$-midpoint of $\bar x_2$ and $\bar x_3$.\\
The same radii can also be used to define the points $x_i\mDef (r_i,y_i)\in X$. Since the distances $\dd(y_i,y_j)$ and $\dd(\bar y_i,\bar y_j)$ and the radii agree, we find that $\tau(x_i,x_j) = \tau(\bar x_i,\bar x_j)$, which means that $\triangle \bar x_1 \bar x_2 \bar x_3\in \R^{1,1}\subset \R^{1,2}$ is a comparison triangle of $\triangle  x_1  x_2  x_3\in X$.\\
A calculation shows that $x_{m \epsilon}$ is a $\mu$-$\epsilon'$ midpoint of $x_2$ and $x_3$ using the notation:
\begin{align*}
    \epsilon' := \epsilon\cdot\frac{r_2 r_3 \sinh (\dd_{23})}{\sqrt{r_2^2+r_3^2 -2 r_2 r_3 \cosh (\dd_{23})}}\ +\ O\left(\epsilon ^2\right)\,,\\
    \mu := \frac{r_2} {r_2+r_3}\quad\text{ ,  }\quad \dd_{ij} = \dd(y_i,y_j)\quad\text{ ,  } \quad\dd_{im}=\dd(y_i,y_{m\epsilon}) \, .
\end{align*}
Then,
\begin{align*}
    &\null \tau(x_2,x_{m\epsilon})- \mu \tau(x_2,x_3) \\
    &= \sqrt{r_2^2+4\frac{r_2^2r_3^2 \cosh^2(\frac{\dd_{23}}{2})}{(r_2^2+r_3^2)^2} - 4\frac{r_2^2 r_3 \cosh(\frac{\dd_{23}}{2})}{r_2+r_3} \cosh(\dd_{2m})} - \\
    &\qquad \qquad \mu \sqrt{r_2^2 + r_3^2- 2 r_2 r_3 \cosh(\dd_{23})}\\
    &\geq \mu\Big(\sqrt{(r_2+r_3)^2 + 4 r_3^2 \cosh^2(\frac{\dd_{23}}{2}) - 4 (r_2+r_3) r_3 \cosh(\frac{\dd_{23}}{2}) \cosh(\frac{1}{2}\dd_{23}+\epsilon)} \\
    &\qquad \qquad - \sqrt{r_2^2 + r_3^2- 2 r_2 r_3 \cosh(\dd_{23})}\Big)\\
    &\geq \mu\Bigg(\Big(r_2^2+3r_3^2 + 2 r_3^2 \cosh(\dd_{23}) - 4 (r_2+r_3) r_3  \big(\frac{\cosh(\dd_{23})+1}{2} \cosh(\epsilon) \\
    &\qquad \qquad + \frac{\sinh(\dd_{23})}{2} \sinh(\epsilon)\big) \Big)^{1/2}  - \sqrt{r_2^2 + r_3^2- 2 r_2 r_3 \cosh(\dd_{23})}\Bigg)\\
    &
    \left.
    \begin{aligned}
    \geq \mu\Big(    
    \sqrt{r_2^2+3 r_3^2 + 2 r_3^2 \cosh(\dd_{23}) 
     - 2 (r_2+r_3) r_3  \big(\cosh(\dd_{23}+\epsilon) + \cosh(\epsilon)\big)} \\
    \qquad\qquad
    - \sqrt{r_2^2 + r_3^2 - 2 r_2 r_3 \cosh(\dd_{23})} \Big)
    \end{aligned}
    \ \right\}
    \quad = (*)
   \end{align*}
   For $\epsilon=0$, we observe:
   \begin{align*}
   \left.(*)\right|_{\epsilon=0}
   &=\mu\Big(\sqrt{r_2^2+3 r_3^2 + 2 r_3^2 \cosh(\dd_{23})  - 2 (r_2+r_3) r_3  \big(\cosh(\dd_{23}) + 1\big) } \\
   & \qquad \qquad- \sqrt{r_2^2 + r_3^2- 2 r_2 r_3 \cosh(\dd_{23})}\Big)\\
    &= \mu\Big(\sqrt{r_2^2 + r_3^2 - 2 r_2 r_3 \cosh(\dd_{23})   } - \sqrt{r_2^2 + r_3^2- 2 r_2 r_3 \cosh(\dd_{23})}\Big) \\
    &=0\,.\\
   \end{align*}
   The first derivative in $\epsilon$ at $\epsilon=0$ is:
   \begin{align*}
       \left.\partial\epsilon(*)\right|_{\epsilon=0}
       &=\mu \frac{- r_3 (r_2+r_3) \, \sinh(\dd_{23})}{ \sqrt{r_2^2 + r_3^2- 2 r_2 r_3 \cosh(\dd_{23})}}\, .
   \end{align*}
    Hence, expanding $(*)$ in $\epsilon$ in first order at $\epsilon=0$, we obtain
    \begin{equation*}
        \tau(x_2,x_{m\epsilon})- \mu \tau(x_2,x_3) \geq - \epsilon \cdot\frac{r_2 r_3  \sinh (\dd_{23})}{\sqrt{r_2^2+r_3^2 -2 r_2 r_3 \cosh (\dd_{23})}}+O\left(\epsilon ^2\right)\, .
    \end{equation*}
    If we, in the first step, approximate $\dd_{2m} \geq \frac{1}{2}\dd_{23}-\epsilon$ instead of $\dd_{2m} \leq \frac{1}{2}\dd_{23}+\epsilon$, we obtain,
    \begin{equation*}
        |\tau(x_2,x_{m\epsilon})- \mu \tau(x_2,x_3)| \leq \epsilon \cdot\frac{r_2 r_3  \sinh (\dd_{23})}{\sqrt{r_2^2+r_3^2 -2 r_2 r_3 \cosh (\dd_{23})}}+O\left(\epsilon ^2\right)\, .
    \end{equation*}
    By an almost identical calculation, we also get  
    \begin{equation*}
        |\tau(x_{m\epsilon},x_3)- (1-\mu) \tau(x_2,x_3)| \leq  \epsilon \cdot\frac{r_2 r_3  \sinh (\dd_{23})}{\sqrt{r_2^2+r_3^2 -2 r_2 r_3 \cosh (\dd_{23})}}+O\left(\epsilon ^2\right)\, .
    \end{equation*}
    This shows that $x_{m\epsilon}$ is a $\mu-\epsilon'$ midpoint of $x_2$ and $x_3$.\\
Using the $\epsilon'$-$\mu$ upper curvature bound on $X$, we can conclude that for some $f:\R_{\geq0} \to \R_{\geq0}$ with $\lim _{\epsilon'\to 0} f(\epsilon')=0$ we have
\begin{equation}\label{equ:inequ1}
  \sqrt{r_1^2 + r_m^2-2 r_1 r_m \cosh(\dd(y_1,y_{m \epsilon}))} \geq \sqrt{r_1^2 + r_m^2-2 r_1 r_m \cosh(\dd(\bar y_1,\bar y_{m}))} -f(\epsilon').
\end{equation}
Hence,
$$\dd(y_1,y_{m \epsilon}) \leq \dd(\bar y_1,\bar y_{m}) + \tilde{f}(\epsilon)$$ 
for some $\tilde{f}$ with $\lim_{\epsilon\to 0} \tilde{f}(\epsilon) =0$.

In the second part of the proof, we will show that $Y$ is a geodesic space. The argument for this is not new and can be found, for example, in 9.13 of \cite{burago2001course}.

Since $Y$ is a length space, for any two points $y_1, y_2 \in Y$ there exists a sequence $(m_k)$ of $\frac{1}{k}$-midpoints. For $k < l \in \mathbb{N}$, consider the triangle $\triangle y_1 y_2 m_k$ and apply triangle comparison to the triangle $\triangle y_1 y_2 m_l$. Let $\bar{d}_k$ be the midpoint in the comparison triangle $\triangle \bar{y}_1 \bar{m}_k \bar{y}_2$. Then we have
$$\dd(m_k, m_l) \leq \dd(m_k, \bar{d}_k) + f\left(\tfrac{1}{l}\right),$$
where $f\left(\tfrac{1}{l}\right) \to 0$ as $l \to \infty$. It remains to estimate $\dd(m_k, \bar{d}_k)$. Observe that $m_k$ lies in the intersection of the two balls of radius $\frac{1}{k} + \frac{1}{2}\dd(y_1, y_2)$ centred at $y_1$ and $y_2$. Since $\dd(y_1, y_2)$ is fixed and $\frac{1}{k} \to 0$, the diameter of this intersection goes to zero, hence $\dd(m_k, \bar{d}_k) \to 0$ as $k \to \infty$.

Thus, the sequence $(m_k)$ is Cauchy. Since $Y$ is complete, there exists a limit point $ \lim_{k\to\infty} m_k = m \in Y$, which by construction lies at equal distance from $y_1$ and $y_2$, and hence is a true midpoint. This proves that $Y$ is a geodesic space.
\begin{center}
    \begin{tikzpicture} 
        \foreach \n in {1,2,...,20} 
        \draw[thin] (0,0) to[out=90.0 / \n, in=180 - 90.0 / \n] (4,0);
         \fill (0,0) circle (1.5pt);
                \node[left] at (0,0) {$y_1$};
        \fill (4,0) circle (1.5pt);
                \node[right] at (4,0) {$y_2$};
        \fill[red] (2,0.82) circle (1.5pt);
                \node[above right] at (2,0.75) {\textcolor{red}{$m_k$}};
        \fill[red] (2,0.45) circle (1.5pt);
                \node[above right] at (2,0.25) {\textcolor{red}{$m_l$}};
\end{tikzpicture}
\end{center}

\end{proof}

\begin{theorem} \label{thm:SoDcba-1}
    Let \Xll be a \LpLS \, satisfying the conditions of \Cref{def-log} and timelike sectional curvature bounded below by $K \in \mathbb{R}$, then the space of directions at every point $p \in X$, $\Sigma_p^+$ has sectional curvature bounded above by $-1$.
\end{theorem}

\begin{proof}
    The proof is a simple application of the main results of the paper. 
    By \Cref{them:SoD_lengthspace}, the future space of directions $\Sigma_p^+$ is a length space and the future tangent cone $T_p^+$ is a \LpLS. Now, by \Cref{lem:TConeflat}, $T_p^+$ has curvature bounded above by $0$ in the sense of the Lorentzian four point condition (\Cref{def:LorfourPts}). This is equivalent to $T_p^+$ having curvature bounded above by $0$ in the one-sided $\epsilon-\mu$ triangle condition by \Cref{lem:curvature_equivalence}. Finally, applying \Cref{lem:Ycba-1} to $T_p^+ = Con(\Sigma_p^+)$ implies that $\Sigma_p^+$ is a geodesic space of curvature bounded above by $-1$. 
\end{proof}

\newpage
\printbibliography

\end{document}